\documentclass[11pt]{article}

\usepackage[letterpaper, hmargin=1.3in, top=1in, bottom=1.3in, footskip=1in]{geometry}
\usepackage{float}
\usepackage{titlesec}
\titleformat{\subsection}[runin]{\normalfont\bfseries}{\thesubsection.}{.5em}{}[.]\titlespacing{\subsection}{0pt}{2ex plus .1ex minus .2ex}{.8em}
\titleformat{\subsubsection}[runin]{\normalfont\itshape}{\thesubsubsection.}{.3em}{}[.]\titlespacing{\subsubsection}{0pt}{1ex plus .1ex minus .2ex}{.5em}

\usepackage[labelfont=sc,font=small,labelsep=period]{caption}
\usepackage{amsmath} %[intlimits]
\usepackage{amssymb}
\usepackage{amsfonts}
\usepackage{latexsym}
\usepackage{amsthm}
\usepackage{amsxtra}
\usepackage{amscd}
\usepackage{bbm}
\usepackage{mathrsfs}
\usepackage{bm}
\usepackage{xcolor}
\usepackage{subcaption}
\usepackage[toc]{appendix}
\usepackage[utf8]{inputenc}
\usepackage{fancyhdr}

\usepackage{graphicx, color}

\definecolor{darkred}{rgb}{0.9,0,0.3}
\definecolor{darkblue}{rgb}{0,0.3,0.9}

\usepackage{ifthen}
\def\comment#1{\ifthenelse{\isodd{\value{page}}}{\marginpar{\raggedright\scriptsize{\textcolor{darkred}{#1}}}}{\marginpar{\raggedleft\scriptsize{\textcolor{darkred}{#1}}}}}  

\usepackage[nottoc,notlof,notlot]{tocbibind}
\usepackage{cite} %Enabling `[1-4]` - style citing

\usepackage{hyperref}
\hypersetup{
  pdftitle={An asymptotic radius of convergence for the Loewner equation and simulation of SLE traces via splitting},
  pdfauthor={James Foster and Vlad Margarint},
  colorlinks=false,
}

\numberwithin{equation}{section}
\numberwithin{figure}{section}
\theoremstyle{plain} %plain, definition, remark
\newtheorem{theorem}{Theorem}[section]
\newtheorem*{theorem*}{Theorem}
\newtheorem{lemma}[theorem]{Lemma}
\newtheorem*{lemma*}{Lemma}

\newtheorem*{corollary*}{Corollary}
\newtheorem{proposition}[theorem]{Proposition}
\newtheorem*{proposition*}{Proposition}
\newtheorem{definition}[theorem]{Definition}
\newtheorem*{definition*}{Definition}

\newtheorem*{conjecture*}{Conjecture}

\theoremstyle{definition} %plain, definition, remark

\newtheorem*{example*}{Example}
\newtheorem{remark}[theorem]{Remark}
\newtheorem*{remark*}{Remark}

\newcommand{\E}{\mathbb{E}}

\renewcommand{\leq}{\leqslant}
\renewcommand{\geq}{\geqslant}
\renewcommand{\epsilon}{\varepsilon}

\DeclareMathOperator{\mesh}{mesh}

\title{An asymptotic radius of convergence for the Loewner equation and simulation of $SLE_k$ traces via splitting}
\author{James Foster, Terry Lyons and Vlad Margarint}

\newcommand\shorttitle{A radius of convergence and simulation of the Loewner equation}
\newcommand\authors{J. Foster, T. Lyons and V. Margarint}

\begin{document}

\maketitle

\begin{abstract}
In this paper, we shall study the convergence of Taylor approximations for the
backward Loewner differential equation (driven by Brownian motion) near the origin. 
More concretely, whenever the initial condition of the backward Loewner equation (which lies in the upper half plane) is small and has the form  $Z_{0} = \epsilon i$, we show these
approximations exhibit an $O(\epsilon)$ error provided the time horizon is $\epsilon^{2+\delta}$ for $\delta > 0$.
Statements of this theorem will be given using both rough path and $L^{2}(\mathbb{P})$ estimates.
Furthermore, over the time horizon of $\epsilon^{2-\delta}$, we shall see that ``higher degree'' terms within the Taylor expansion become larger than ``lower degree'' terms for small $\epsilon$.
In this sense, the time horizon on which approximations are accurate scales like $\epsilon^{2}$.
This scaling comes naturally from the Loewner equation when growing vector field
derivatives are balanced against decaying iterated integrals of the Brownian motion.
As well as being of theoretical interest, this scaling may be used as a guiding principle for developing adaptive step size strategies which perform efficiently near the origin.
In addition, this result highlights the limitations of using stochastic Taylor methods
(such as the Euler-Maruyama and Milstein methods) for approximating $SLE_{\kappa}$ traces. 
Due to the analytically tractable vector fields of the Loewner equation, we will show
Ninomiya-Victoir (or Strang) splitting is particularly well suited for SLE simulation.
As the singularity at the origin can lead to large numerical errors, we shall employ
the adaptive step size proposed in \cite{kennedy2009slesim} to discretize $SLE_{\kappa}$ traces using this splitting. We believe that the Ninomiya-Victoir scheme is the first high order numerical method that has been successfully applied to $SLE_{\kappa}$ traces.
\end{abstract}

\section{Introduction}
Rough Path Theory was first introduced in 1998 by Terry Lyons in \cite{lyons1998differential}. The theory provides a deterministic platform to study stochastic differential equations which extends both Young's integration and stochastic integration theory beyond regular functions and semimartingales. In addition, rough path theory provides a methodology for constructing solutions to differential equations driven by paths that are not of bounded variation but have controlled roughness. Step by step, we introduce the ingredients and terminology necessary to characterize the roughness of a path and to give precise meaning to natural objects that appear in the study of rough paths. The Schramm-Loewner evolution, or $SLE_\kappa$ is a one parameter family of random planar fractal curves introduced by Schramm in \cite{schramm2000scaling}, that are proved to describe scaling limits of a number of discrete models appearing in planar statistical physics. For instance, it was proved in \cite{lawler2011conformal}  that the scaling limit of loop erased random walk (with the loops erased in a chronological order) converges in the scaling limit to $SLE_{\kappa}$ with  $\kappa = 2\,.$ In addition, other two dimensional discrete models from statistical mechanics including Ising model cluster boundaries, Gaussian free field interfaces, percolation on the triangular lattice at critical probability, and uniform spanning trees were proved to converge in the scaling limit to $SLE_{\kappa}$ for values of $\kappa=3,$ $\kappa=4,$ $\kappa=6$ and $\kappa=8$ respectively in the series of works  \cite{lawler2011conformal}, \cite{schramm2009contour}, \cite{smirnov2001critical}  and \cite{smirnov2010conformal}. For a detailed study of SLE theory, we refer the reader to  \cite{lawler2008conformally} and \cite{rohde2011basic}.

Throughout the years, a number of papers have been written at the interface between the aforementioned domains. The paper of Brent Werness \cite{werness2012regularity} defines the expected signature for the $SLE_{\kappa}$ traces, that is the expected values of iterated integrals of the path against itself. This approach provides new ideas about how one can use a version of Green's formula for rough paths and a certain observable for $SLE_{\kappa}$ to compute the first three elements of the expected signature of $SLE_{\kappa}$ in the regime $ \kappa \in [0,4]\,.$ An extension to this computation is provided in \cite{boedihardjo2014uniqueness}, where the authors show ways of computing the fourth grading of the signature (and do it explicitly for $SLE_{8/3}$, where the required observable is known), along with several other parts of the higher grading. From a different perspective, in \cite{friz2017existence} the authors question the existence of the trace for a general class of processes (such as semimartingales) as a driving function in the Loewner differential equation. These ideas are also developed with a Rough Path flavour. More recently, Peter Friz with Huy Tran in \cite{friz2017regularity} revisited the regularity of the $SLE_{\kappa}$ traces and obtained a clear result using Besov spaces type analysis. In \cite{shekhar2019remarks}, Atul Shekhar, Huy Tran and Yilin Wang studied the continuity of the traces generated by Loewner chains driven by bounded variation drivers.

In this paper, we shall use techniques provided by rough path theory to study certain
Taylor approximations of the Loewner differential equation (driven by Brownian motion).
Our motivation is that simulations of $SLE_{\kappa}$ traces are usually done in a pathwise sense.
In particular, as the successful approach of Tom Kennedy in \cite{kennedy2009slesim} discretizes $SLE_\kappa$ with
adaptive step sizes, it is reasonable to investigate the time scales on which the Loewner differential equation can
be well approximated by strong (or pathwise) numerical methods.
The main result of this paper identifies a natural scaling between the initial condition of the
backward Loewner differential equation and the largest time horizon at which stochastic
Taylor approximations are reasonable (i.e. produce small errors). Since this diffusion starts (in a limiting sense) at a singularity at zero in $SLE_\kappa$ theory, we wish to understand
the solution's dynamics when the size of the initial condition $\epsilon = |Z_{0}| > 0$ is small.
We will show that over the small time horizon $\epsilon^{2+\delta}$ for $\delta > 0$, a one-step Taylor approximation of the backward Loewner diffusion will exhibit a local error that is $O(\epsilon)$. Moreover, on a larger time scale of $\epsilon^{2-\delta}$, we will see that the vector field derivatives grow
faster asymptotically than corresponding iterated integrals of Brownian motion and time.

In other words, we provide an analysis of the edge case in Theorem \ref{theoremTaylor}. That is, the critical case whereby over longer time horizons, the Taylor approximation contains large terms (for degree large enough $deg>1/\delta$, for fixed $\delta$ and small $\epsilon$) even when the true Loewner dynamics is small.  In particular, this means that they can't be close to each other -- which is expressed in Remark 3.2. Moreover, in Remark 5.1 we discuss another possible upper bound on the difference. Ultimately, this result highlights the limitations of using stochastic Taylor methods (such as the Euler-Maruyama and Milstein methods) for approximating solutions of the Loewner equation.

Finally, we shall discretize $SLE_{\kappa}$ traces using Ninomiya-Victoir (or Strang) splitting. 
Although this approach will also be limited in terms of pathwise accuracy, it will have additional advantages when one considers the weak convergence of the numerical solution.
This is because the Ninomiya-Victoir scheme achieves a $O(h^{2})$ weak convergence rate (where $h$ denotes the step size used) for SDEs that have sufficiently regular vector fields.
To the best of our knowledge, the methods proposed by Kennedy in \cite{kennedy2009slesim} and analysed by Tran \cite{tran2015convergence} have not been shown to have such high order weak convergence. 
Moreover in our case, this method preserves second and fourth moments of the backward Loewner diffusion.
Further strong convergence results for this splitting of $SLE_{\kappa}$ traces are established in \cite{chen2021sleconvergence}. 
Example code for this method can be found at \href{https://github.com/james-m-foster/sle-simulation}{github.com/james-m-foster/sle-simulation}.\medbreak

\subsection*{Funding and acknowledgements}
The first author was supported by the Department of Mathematical Sciences at the University of Bath and the DataSig programme under the EPSRC grant S026347/1. The second author was supported by the DataSig programme and Alan Turing Institute under EPSRC grant EP/N510129/1. The last author would like to acknowledge the support of ERC (Grant Agreement No.291244 Esig) between 2015-2017 at the OMI Institute, EPSRC 1657722 between 2015-2018, Oxford Mathematical Department Grant and the EPSRC Grant EP/M002896/1 between 2018-2019. In addition, VM acknowledges the support of the NYU-ECNU Institute of Mathematical Sciences at NYU Shanghai. We would also like to thank Ilya Chevyrev, Dmitry Belyaev, Danyu Yang and Weijun Xi for useful suggestions and reading previous versions of this manuscript. 

%The first and second authors were both supported by the DataSig programme under the ESPRC grant S026347/1 and by the Alan Turing Institute under the EPSRC grant EP/N510129/1.  The last author was supported by the NYU-ECNU Institute of Mathematical Sciences at NYU Shanghai. In addition, we would like to thank Ilya Chevyrev, Dmitry Belyaev, Danyu Yang and Weijun Xi for useful suggestions and reading previous versions of this manuscript. 

%The first author is supported by the EPSRC Grant EP/N509711/1.

\section{Rough Path Theory overview}

In this section, we shall highlight the key aspects of Rough Path theory that are utilized within the paper. In particular, more detailed accounts are given in textbooks such as \cite{friz2010multidimensional}.\medbreak

Let $X_{[s,t]}$ denote the restriction of the path $X: [0,T] \to V$ to the compact interval $[s,t]$, where $V$ is a finite dimensional real vector space. We introduce the notion of $p$-variation. 
\begin{definition}
Let $V$ denote a finite dimensional real vector space with dimension $d$ and basis vectors $e_1,\cdots, e_d\,$. The $p$-variation of a path $X: [0,T] \to V $ over $[s,t]$ is defined by 
\begin{align*}
\big\|X_{[s,t]}\big\|_{p-var} \; := \; \sup_{\mathcal{D}=(t_0, t_1,\cdots,t_n) \subset [s,t]}\left( \sum\limits_{i=0}^{n-1}\left\|X_{t_{i+1}} - X_{t_i}\right\|^p \right)^{\frac{1}{p}}\,,
\end{align*}
where the supremum is taken over all finite partitions of the interval $[s,t]\,.$
\end{definition}
Throughout the paper we use the notation $X_{s,t}= X_t-X_s$ for increments of a path.
For $T>0$, let us define $\Delta_T = \big\{ (s,t)\,|\,0 \leq s\leq t \leq T\big\}$
and the important notion of control:
\begin{definition}
 A control on $[0, T]$ is a non-negative continuous function $\omega : \Delta_T \to [0, \infty)$ for which\vspace{-1.5mm}
\begin{align*}
\omega(s,t)+\omega(t,u) \leq \omega(s,u),
\end{align*}
for all $\,0 \leq s \leq t \leq u \leq T, $ and $\omega (t,t) = 0,$  for all $t \in [0,T]\,.$
\end{definition}
We introduce also the following spaces required to define rough paths.
\begin{definition}
Let $T((V)):= \left\{\boldsymbol{a} = (a_{0}, \,a_{1},\, \cdots\,) : a_{n}\in V^{\otimes n}\,\,\,\forall n\geq 0\right\}$ denote the set of formal series of tensors of $V\,.$
\end{definition}
\begin{definition}
The tensor algebra $T(V):= \bigoplus_{k\geq 0}V^{\otimes_k}$ is the infinite sum of all tensor products of $V\,.$
\end{definition}

Suppose that $e_1, e_2, \cdots, e_d $ is a basis for $V\,.$ Then the space $V^{\otimes k}$ is a $d^k$-dimensional vector space that has basis elements of the form $(e_{i_1} \otimes\, e_{i_2}\otimes \cdots \otimes \,e_{i_k})_{(i_1,\cdots,i_k) \in \{ 1,\cdots, d \}^k} \,.$ We store the indices $(i_1, \cdots , i_k) \in \{1,2, \cdots, d \} ^{k}$ in a multi-index $I$ and let $ e_{I}=e_{i_1} \otimes\cdots\otimes e_{i_k}\,.$
The metric $||\cdot||$ on $T((V))$ is the projective norm defined for $$x= \sum_{|I|=k}\lambda_I e_I \in V^{\otimes k}$$ via $$||x||=\sum_{|I|=k}|\lambda_I|.$$  
%Thus, the bound $||X^i_{s,t}|| \leq \frac{{w(s,t)}^{i/p}}{\beta(\frac{i}{p})!}, \hspace{2mm} \forall i \geq 1,\hspace{2mm} \forall (s,t) \in \Delta_T,$ gives control on the sum of $i$-iterated integrals.
%We collect all the iterated integrals in the following way.
We consider for $$X: \Delta_T \to T((\mathbb{R}))$$ the collection of iterated integrals as $$(s,t) \to \bold{X}_{s,t}= (1, X_{s,t}^1, \cdots, X_{s,t}^n, \cdots ) \in T((V)).$$  We call the collections of iterated integrals the signature of the path $X$.

We now define the notion of $\textit{multiplicative functional}$.
\begin{definition}
Let $n \geq 1$ be an integer and let $X : \Delta_T \to T^{(n)}(V)$ be a continuous map. Denote by $X_{s,t}$ the image of the interval $(s,t)$ by $X\,,$ and write
$$X_{s,t}=(X_{s,t}^{0}, \cdots, X_{s,t}^{n}) \in \mathbb{R}\oplus V\oplus V^{\otimes 2} \cdots \oplus V^{\otimes n}\,.$$
The function $X$ is called multiplicative functional of degree $n$ in $V$ if $X_{s,t}^0=1$ and for all $(s,t) \in \Delta_t$ it satisfies the so-called ``Chen relation''
$$X_{s,u} \otimes X_{u,t}= X_{s,t} \hspace{2mm} \forall s,u,t \in [0,T]\,.$$
\end{definition}
We will use the notion of $p$-rough path that we define in the following. 
\begin{definition}\label{def1}
A $p$-rough path of degree $n$ is a map $X : \Delta_T \to T^{(n)}(V)$ which satisfies Chen's identity $X_{s,u}\otimes X_{u,s}=X_{s,t}$ and the following 'level dependent' analytic bound
$$\|X_{s,t}^{i}\| \leq \frac{w(s,t)^{\frac{i}{p}}}{\beta_p (\frac{i}{p})!} \,,$$
\end{definition}
\noindent
where $y!=\Gamma(y+1)$ whenever $y$ is a positive real number and $\beta_p$ is a positive constant. 

Since the driving rough path $X$ in this paper will be a standard Brownian motion
coupled with time, we will require estimates for rough paths with different homogeneities.
Hence we use the notation of $\Pi$-rough paths introduced by Lajos Gergely Gyurk\'{o} in \cite{gyurko2016roughpaths}.
We shall give the key definitions and theorems for this when there are two homogeneities.

\begin{definition}
Let $V^{0}$ and $V^{1}$ denote two vector spaces with direct sum $V = V^{0}\oplus V^{1}$.
Then for any multi-index $I=(i_{1},\cdots, i_{l})\in\{0,1\}^{\ast}$, we define the following vector space:
\begin{align*}
V^{\otimes I} & := V^{i_{1}}\otimes\cdots\otimes V^{i_{l}}.
\end{align*}
\end{definition}

\begin{definition}\label{degreedef}
The $(p,q)$-degree of a multi-index $I=(i_{1},\cdots, i_{l})\in\{0,1\}^{\ast}$ is defined as
\begin{align*}
\deg_{(p,q)}(I) := \frac{n}{p} + \frac{m}{q},
\end{align*}
where
\begin{align*}
m & := \#\{\,j : i_{j} = 0, i_{j}\in I\},\\[3pt]
n & := \#\{\,j : i_{j} = 1, i_{j}\in I\}.
\end{align*}
For a fixed $k\geq 1$, we say that a multi-index $I=(i_{1},\cdots, i_{l})\in\{0,1\}^{\ast}$ with $\deg_{(p,q)}(I) \leq k$ is $(p,q)$-maximal if there exists $j\in\{0,1\}$ such that $\deg_{(p,q)}(i_{1},\cdots, i_{l}, j) > k$.\medbreak
\noindent
Using the above, we can introduce a ``\,factorial'' function $\Gamma_{(p,q)}$ on the set of multi-indices,
\begin{align*}
\Gamma_{(p,q)}(I) := \left(\frac{n}{p}\right)!\left(\frac{m}{q}\right)!\,,
\end{align*}
where $I\in\{0,1\}^{\ast}$ is the same multi-index and $(\cdot)!$ denotes the standard Gamma function.
\end{definition}
\begin{definition} Using the $(p,q)$-degree, we can define a truncated tensor algebra as
\begin{align*}
T^{((p,\,q),\,k)}(V) & := \bigoplus_{\deg_{(p,q)}(I)\, \leq\, k}V^{\otimes I}\hspace{2mm}\text{for}\hspace{2mm}n,m\geq 0.
\end{align*}
Then for a fixed element $X\in T^{((p,\,q),\,k)}(V)$ and a multi-index $I=(i_{1},\cdots, i_{l})\in\{0,1\}^{\ast}$ with $\deg_{(p,q)}(I)\, \leq\, k$ , we shall denote $X^{I}$ as the projection of $X$ onto its $V^{\otimes I}$ component.
\end{definition}

\begin{definition}
A $(p,q)$-rough path of degree $k$ is a continuous map $X : \Delta_T \to T^{((p,\,q),\,k)}(V)$ which satisfies Chen's identity $X_{s,u}\otimes X_{u,t}=X_{s,t}$ and the ``\,level dependent'' analytic bound
\begin{align*}
\left\|X_{s,t}^{I}\right\| \leq \frac{w(s,t)^{\deg_{(p,q)}(I)}}{\beta\,\Gamma_{(p,q)}(I)} \,,\hspace{5mm}\text{for all}\hspace{2mm}(s,t)\in \Delta_T,
\end{align*}
where $\beta$ is a positive constant and $I\in\{0,1\}^{\ast}$ is any multi-index with $\deg_{(p,q)}(I) \leq k$.
\end{definition}

\begin{theorem}[Theorem 2.6 of \cite{gyurko2016roughpaths}]\label{gyurko}
Let $X$ denote a $(p,q)$-rough path that has degree $1$.
Then for every $k\geq 1$ there exists a unique $(p,q)$-rough path $\widetilde{X}$ of degree $k$ such that
\begin{align*}
X_{s,t}^{I} = \widetilde{X}_{s,t}^{I}
\end{align*}
for any multi-index $I\in\{0,1\}^{\ast}$ with $\deg_{(p,q)}(I) \leq 1$. Thus, we have the following estimate
\begin{align*}
\big\|\widetilde{X}_{s,t}^I\big\| \leq \frac{w(s,t)^{m+\frac{n}{p}}}{\beta\,m!\big(\frac{n}{p}\big)!}\,,
\end{align*}
where $\beta$ is a positive constant and $I=(i_1,\cdots, i_l)\in\{0,1\}^{\ast}$ is any multi-index such that
\begin{align*}
k & \geq m + \frac{n}{p},
\end{align*}
where $m$ and $n$ are the same non-negative integers given by definition (\ref{degreedef})
\end{theorem}
It is well known that a Brownian motion $B_t$ can be enhanced to a $p$-rough path using either It\^{o} or Stratonovich integration (this is detailed in standard textbooks, such as \cite{friz2010multidimensional}).
As the function $t\mapsto t$ has finite variation, this immediately leads to the following theorem:
\begin{theorem}
The ``space-time'' Brownian motion, $X_t := (t,B_t)$, can be enhanced to a
$(1, p)$-rough path in either an It\^{o} or a Stratonovich sense, almost surely, for any $p > 2$.
\end{theorem}

In general, for a Rough Differential Equation of the form 
$$dZ_t=V(Z_t)dX_t,$$ with $Z_0=|z_0|>0$ with $X_t:[0,T] \to \mathbb{R}^d$ a finite $p$-variation path for any $p>2$, we use the following compact notation for the first $r$ terms of the Taylor approximation.

\begin{definition}\label{Tayloraprox}
Let $I=(i_1,..,i_k)$ be a multi-index. Given the continuously differentiable vector fields $(V_1, \cdots, V_d)$ on $\mathbb{R}^{e}$, and a multiplicative functional with finite $p$-variation, $\bold{X} \in T((\mathbb{R}^d))$, we define $$ \mathcal{E}^{r}_{(V)}(Z_0, \bold{X}_{0,t}):=\sum_{I: deg(I) \leq r} V_{i_1}\cdots V_{i_k}\textbf{Id}(Z_0)\bold{X}_{0,t}^{k, i_1,\cdots, i_k}$$
as the increment of the step $r$-truncated Taylor approximation on the interval $[0,t]$.\\
The notation  $V_{i_0} \cdots V_{i_k}\textbf{Id}$ stands for the composition of differential operators associated with the vector fields,
\begin{align*}
V_{i_1}\cdots V_{i_k}\textbf{Id}(Z_0) = \left(\left(V_{i_{1}}\sum_{i=1}^{e}\frac{d}{dx_{i}}\right)\circ \cdots \circ \left(V_{i_{k-1}}\sum_{i=1}^{e}\frac{d}{dx_{i}}\right)\circ V_{i_{k}}\right)\hspace{-1mm}\big(Z_{0}\big),
\end{align*}
and $\bold{X}_{0,t}^{i_1,\cdots, i_k}$ stands for the terms obtained from the iterated integrals $$\int_{0<s_1<..<s_k<t}dX_{s_1}^{i_1}\cdots dX_{s_k}^{i_k}.$$
\end{definition}

\section{Main result}

Let us consider the backward Loewner differential equation driven by Brownian motion. 
\begin{align}
\frac{\partial h_{t}(z)}{\partial t}& = \frac{-2}{h_{t}(z) - \sqrt{\kappa}\, B_{t}}\,,\\
h_{0}(z) & = z\,.\nonumber
\end{align}

By performing the identification $h_t(z)-B_t=Z_t$, we obtain the following dynamics in $\mathbb{H}\,,$ that we consider throughout this section
\begin{align}\label{backwardeq2}
dZ_t=\frac{-2/\kappa}{Z_t}dt+dB_t, \hspace{5mm}Z_0=z_0 \in \mathbb H.
\end{align}

Let $\epsilon>0$. Let us consider the starting point of the backward Loewner differential equation $z_0 \in \mathbb{H}$ with  $|z_0|=\sqrt{x_0^2+y_0^2}=\epsilon$.

Considering the equation for the imaginary part of the dynamics under the backward Loewner differential equation in the upper half-plane $dY_t=\frac{2Y_t}{X_t^2+Y_t^2}dt$, we obtain that the imaginary part is increasing, a.s.. Thus, the backward Loewner differential equation starting from $\epsilon>0$ is a Rough Differential Equation, i.e. we have that the backward Loewner Differential Equation can be written as 
$$dZ_t=V(Z_t)dX_t,$$
where $V(Z)=(V_0(Z), V_1(Z))$ with $V_0(Z)=-\frac{2/\kappa}{Z}\frac{d}{dz}$ and $V_1(Z)=\frac{d}{dx}$ that is driven by the space-time Brownian motion $X_t=(t, B_t)$. 
Thus from Definition \ref{Tayloraprox}, we have a truncated Taylor approximation associated with the backward Loewner Differential Equation. 

Our main result, along with an important remark are given below.

\begin{theorem}\label{theoremTaylor}
The one-step $r$-truncated Taylor approximation for the backward Loewner differential equation started from $Z_0 = \epsilon i$ with $\epsilon>0$  admits an $O(\epsilon)$ error almost surely over the time horizon $0 \leq t \leq \epsilon^{2+\delta}$, for any $\delta>0$ and sufficiently small $\epsilon$. In other words,
\begin{align}\label{mainresult}
\left|\,Z_{t} - \mathcal{E}^{r}_{(V)}(Z_0, \bold{X}_{0,t})\,\right|\leq C\,\epsilon,\hspace{2.5mm}\forall t\in\big[0,\epsilon^{2+\delta}\,\big],
\end{align}
where $\epsilon$ is sufficiently small and $C$ is an a.s. finite constant that depends only on $r$ and $\delta$. 
\end{theorem}
\begin{remark}
It is natural to investigate what happens to the error when the diffusion process and truncated Taylor approximation are taken over a larger time horizon of $\epsilon^{2-\delta}$.
In this case, one can quantify the size of vector field derivatives and iterated integrals as
\begin{align*}
|V_{i_1 }\cdots V_{i_l}\textbf{Id}(Z_0)| & = O\left(\frac{1}{\epsilon^{2m-1+n}}\right),\\
\left\|\int_{0\,<\,s_1\,<\,\cdots\,<\,s_l\,<\,\epsilon^{2-\delta}}dX_{s_1}^{i_1}\cdots dX_{s_l}^{i_l}\right\|_{L^{2}(\mathbb{P})} & = O\left(\epsilon^{(2-\delta)(m+\frac{1}{2}n)}\right),
\end{align*}
where $I=(i_{1},\cdots, i_{k})\in\{0,1\}^{\ast}$ is a a multi-index, $X_t=(t, B_t)$ is a space-time Brownian motion and $(m, n)$ are the non-negative integers given in Definition \ref{degreedef} and Theorem \ref{gyurko}.
We shall see in Section 6 that the above iterated integral can be estimated in an $L^{2}(\mathbb{P})$ sense using the scaling properties of space-time Brownian motion. This indicates that on a time horizon of $\epsilon^{2-\delta}$, the ``higher degree'' terms in a Taylor approximation increase since
\begin{align*}
\left\|\,V_{i_1 }\cdots V_{i_l}\textbf{Id}(Z_0)\int_{0\,<\,s_1\,<\,\cdots\,<\,s_l\,<\,\epsilon^{2-\delta}}dX_{s_1}^{i_1}\cdots dX_{s_l}^{i_l}\,\right\|_{L^{2}(\mathbb{P})} & = O\left(\epsilon^{1-\delta\deg_{(1,2)}(I)}\right).
\end{align*}
As a result, such a Taylor expansion cannot converge absolutely (in an $L^{2}(\mathbb{P})$ sense).
Moreover, to extend (\ref{mainresult}) to a longer time horizon of $\epsilon^{2-\delta}$, we would have to estimate the error of a Taylor expansion where the leading terms in the remainder are larger than $O(\epsilon)$.
That is, $\epsilon^{2}$ is an ``asymptotic'' radius of convergence for the backward Loewner equation.\medbreak

To summarize, we have identified the time span whereby the difference between Loewner dynamics and its Taylor approximation is small, and have shown that this is no longer the case if we extend that time interval. Alternatively, this can also be viewed as the longest time horizon on which Taylor approximations of Loewner dynamics remain well behaved.
\end{remark}

%\begin{remark}
%Note that since the imaginary part of the backward Loewner differential equation satisfies the equation
%$dY_t=\frac{2Y_t}{X_t^2+Y_t^2}dt$, we obtain that the imaginary part is always increasing. Thus, we can concatenate the estimates and obtain from the one-step Taylor approximation a multi-step Taylor approximation. Indeed since the imaginary part increases, the contribution of the absolute value of the vector fields in the remainder becomes smaller, and thus the same time step works. We refer to Section $10.3.5$ in \cite{friz2010multidimensional} for more details. 
%\end{remark}

\section{Asymptotic growth of the vector fields}

We consider backward Loewner Differential Equation
\begin{align}\label{backwardeq}
dZ_t=\frac{-2/\kappa}{Z_t}dt+dB_t, \hspace{5mm}Z_0=z_0 \in \mathbb H,
\end{align}
with $|z_0|=\epsilon>0$ and with the vector fields $V_0(z)=\frac{-2/\kappa}{z}\frac{d}{dz}$ and $V_1(z)=\frac{d}{dx}$.

In order to prove the main result, we first prove the following lemmas in this section.

\begin{lemma}\label{Lemmacombinations}
At fixed level $r$ there are $2^r$ terms obtained from all the possible ways of composing the vector fields $V_0=\frac{-2/\kappa}{Z}\frac{d}{dz}$ and $V_1=\frac{d}{dx}$.
\end{lemma}

\begin{proof}
We prove this using induction. 
For $r=1$, there are $2=2^1$ possible terms obtained from either of the vector fields. 
For $r=2$, the possible compositions are $V_1 \circ V_0$, $V_1\circ V_1$ , $V_1 \circ V_0$ and $V_0\circ V_1$ given $2^2$ possibilities. 
Let us assume that at level $k>0$ there are $2^k$ possible combinations. 
To obtain all the possible compositions at level $k+1$, we have to consider $V_0 \circ V^{k}$ and $V_1 \circ V^{k}$ where $V^{k}$ are all the possible compositions at level $k$. Thus, at level $k+1$ we obtain in total $2^k+2^k=2^{k+1}$ possibilities, and the argument follows by induction.
\end{proof}

\begin{lemma}\label{lemmavectorfields}
Let $r=m+n$, with $n$ being the number of $dB_t$ entries and $m$ being the number of $dt$ entries in the $r$-level iterated integral $\int_{0<s_1<\cdots<s_r<t} dX_{s_1}\cdots dX_{s_r}$. Then, for $Z_0=\epsilon i$ with $\epsilon > 0$, we have that 
$$|V_{i_1 }\cdots V_{i_r}\textbf{Id}(Z_0)|=O\left(\frac{1}{\epsilon^{2m-1+n}}\right).$$

\end{lemma}
\begin{proof}

Given the format of the backward Loewner differential equation we have that the vector fields that can appear are either $V_0=\frac{-2/\kappa}{Z}\frac{d}{d z}$ or $V_1=\frac{d}{d x}$. Note that from the Cauchy-Riemann equations we deduce that for complex differentiable functions the linear differential operators $\frac{d}{dx}$ and $\frac{d}{dz}$ are equivalent when acting on complex differentiable functions, where $\frac{d}{dz}$ denotes the complex differentiation.
For fixed values of $m$ and $n=r-m$ we have, by definition $m$ time entries in the iterated integrals and $m$ times the vector field  $V_0=\frac{-2/\kappa}{Z}\frac{d}{dz}$ and $n$ of the entries $dB_t$ together with $n$ times the vector field $V_1=\frac{d}{dx}$.
We also note that in order for $ V_{i_1 }\cdots V_{i_r}\textbf{Id}(z)$ to be non zero then $V_{i_r}=\frac{1}{Z}$, otherwise $V_{i_r}=1$ and applying any other choice of $V_{i_{r-1}}$ will give the derivative of a constant that is zero. 

Then, we provide the following rules when considering the composition of vector fields $ V_{i_1 }\cdots V_{i_r}\textbf{Id}(z)$ (up to some absolute constants that change, but we avoid keeping their dependence in our analysis). These rules are specific to the structure of the vector fields of the Loewner differential equation and they give a way to transform the composition of the differential operators associated with vector fields in the left into multiplication with the function on the right up to some constants (that we do not keep track of since they do not influence the analysis)
\begin{itemize}
\item $\frac{-2/\kappa}{Z}\frac{d}{d z}\longleftrightarrow \frac{2/\kappa}{Z^2}$

\item $\frac{d}{dx} \longleftrightarrow  \frac{-1}{Z}$
\end{itemize}
in the following sense: $V_0\circ V^k= \frac{2/\kappa}{Z^2} f(Z)$  and  $V_1\circ V^k= \frac{-1}{Z} f(Z) $, where we have that $f(z)=V_{i_1} \circ V_{i_2} \circ \cdots \circ V_{i_k}\circ\textbf{Id}(Z) $ for any $k>0$.

To illustrate this, we consider $$V_0\circ V_0 \circ \bold{Id}(Z)= \frac{1}{Z}\frac{d}{d z}\frac{-1}{Z}.$$ Then, $\frac{d}{d z}\frac{1}{Z}=\frac{1}{Z^2}$ and $V_0\circ V_0 \circ \bold{Id}(Z)=\frac{-1}{Z^3}$ up to some constants. In general, the analysis is similar. Indeed, the result of the composition is obtained by iteration using either of the vector fields applied to the initial $\frac{-2/\kappa}{Z}$ value (since otherwise we would obtain zero since we apply differential operators to a constant), and since the differential operators either act by differentiation or by differentiation and multiplication with $\frac{1}{Z}$ up to some constants, the rules hold.
Using these rules, independent of the order of applications of the vector fields we have that for $Z_0=\epsilon i$,
$$|V_{i_1 }\cdots V_{i_r}\textbf{Id}(Z_0)|=O\left(\frac{1}{\epsilon^n}\frac{1}{\epsilon^{2m-1}}\right)=O\left(\frac{1}{\epsilon^{2m-1+n}}\right).$$
\end{proof}

\section{Proof of Theorem \ref{theoremTaylor}}

Let $r \geq 1$ and $\delta > 0$ be fixed. Then we are interested in estimating the absolute value of the truncated Taylor approximation remainder for $t\in\big[0, \epsilon^{2+\delta}\big)$,
\begin{align}\label{remainder}
R_{r}(t, Z_t, \bold{X}) = Z_{t} - \mathcal{E}_{(V)}^{r}(Z_0, \bold{X}_{0,t}),
\end{align}
where $Z_0 = \epsilon i$ with $\epsilon > 0$.\medbreak
\noindent
Our first step is obtain bounds on the truncated Taylor approximation itself. Recall that
\begin{align*}
\mathcal{E}_{(V)}^{r}(Z_0, \bold{X}_{0,t}) = \sum_{I: deg(I) \leq r} V_{i_1}\cdots V_{i_k}\textbf{Id}(Z_0)\bold{X}_{0,t}^{i_1,\cdots, i_k}.
\end{align*}
Therefore, by the triangle inequality and definition of supremum norm, it follows that
\begin{align*}
\left|\mathcal{E}_{(V)}^{r}(Z_0, \bold{X}_{0,t})\right| \leq \sum_{I: deg(I) \leq r} \|V_{i_1}\cdots V_{i_k}\textbf{Id}\|_{\infty}\,\big|\bold{X}_{0,t}^{i_1,\cdots, i_k}\big|.
\end{align*}

Since the space-time Brownian motion $X_t=(t, B_t)$ is a $(1,p)$-rough path with $p>2$, we can apply the estimate for iterated integrals given by the extension Theorem \ref{gyurko}.
Thus for any multi-index $I\in\{0,1\}^{\ast}$, there exists an a.s. finite constant $C_I > 0$ such that
\begin{align*}
\big|\bold{X}_{0,t}^{I}\big| \leq C_{I}\,\omega(0,t)^{m+\frac{n}{p}},
\end{align*}
where
\begin{align*}
m & = \#\{\,j : i_j = 0, i_{j}\in I\},\\[3pt]
n & = \#\{\,j : i_j = 1, i_{j}\in I\}.
\end{align*}
On the other hand, it was shown previously that the vector field derivatives grow as
\begin{align*}
\|V_{i_1}\cdots V_{i_k}\text{Id}\|_{\infty} = O\left(\frac{1}{\epsilon^{2m+n-1}}\right).
\end{align*}
Hence, for each multi-index $I$, there exist an almost surely finite constant $\widetilde{C}_{I}$ such that 
\begin{align*}
\|V_{i_1}\cdots V_{i_k}\textbf{Id}\|_{\infty}\,\big|\bold{X}_{0,t}^{i_1,\cdots, i_k}\big| & \leq \widetilde{C}_{I}\,\frac{1}{\epsilon^{2m+n-1}}\,t^{m+\frac{n}{p}}\\[3pt]
& \leq \widetilde{C}_{I}\,\frac{1}{\epsilon^{2m+n-1}}\,\epsilon^{(2+\delta)\left(m+\frac{n}{p}\right)}\\[3pt]
& \leq \widetilde{C}_{I}\,\epsilon^{1 + \delta m + \big((2+\delta)\frac{1}{p} - 1\big)n},
\end{align*}
This immediately gives
\begin{align*}
\|V_{i_1}\cdots V_{i_k}\textbf{Id}\|_{\infty}\,\big|\bold{X}_{0,t}^{i_1,\cdots, i_k}\big| \leq \widetilde{C}_{I}\,\epsilon,
\end{align*}
provided that $2 < p \leq \max\big(\frac{n}{n-m\delta}\,, 1\big)\big(2+\delta\big)$. Hence, by setting $p = 2 + \delta$, we have that
\begin{align*}
\left|\mathcal{E}_{(V)}^{r}(Z_0, \bold{X}_{0,t})\right| \leq \widetilde{C}\,\epsilon,\hspace{2.5mm}\forall t\in\big[0,\epsilon^{2+\delta}\,\big),
\end{align*}
where the constant $\widetilde{C} > 0$ is almost surely finite. Similarly, we have that for $t\in\big[0,\epsilon^{2+\delta}\,\big)$,
\begin{align*}
\big|Z_{t}\big| & = \left|\,\int_{0}^{t}-\frac{2/\kappa}{Z_{s}}\,ds + B_{t}\,\right|\\[3pt]
& \leq \frac{2}{\kappa}\int_{0}^{t}\frac{1}{|Z_{s}|}\,ds + |B_{t}|\\[3pt]
& \leq \frac{2}{\kappa}\cdot\frac{1}{\epsilon}\cdot t + C_{(1)}\,t^{\frac{1}{2+\delta}}\\[3pt]
& \leq \frac{2}{\kappa}\,\epsilon^{1+\delta} + C_{(1)}\,\epsilon\,,
\end{align*}
where we have used the facts that the imaginary part of $|Z_{t}|$ is increasing and $|Z_{0}|= \epsilon$.
By combining these estimates, we have the desired result that for a fixed $r\geq 1$ and $\delta > 0$,
\begin{align*}
\big|R_{r}(t, Z_t, \bold{X})\big| \leq C\,\epsilon,\hspace{2.5mm}\forall t\in\big[0,\epsilon^{2+\delta}\,\big),
\end{align*}
where $\epsilon$ is sufficiently small and $C$ is an a.s. finite constant that depends only on $r$ and $\delta$.

\begin{remark}
Alternatively, we could have used a rough Taylor expansion (Theorem 1.1 in \cite{yang2015taylor}) to estimate the approximation error $R$. This would give an estimate with the form:
\begin{align*}
\left|Z_{t} - \mathcal{E}_{(V)}^{r}(Z_0, \bold{X}_{0,t})\right| \leq C\sum_{\substack{I = (i_{1},\cdots, i_{k})\in\{0,1\}^{\ast} \\ \deg_{(1,p)}(I) \leq r\\ I\,\text{maximal}}}\big\|V_{i_{1}}\cdots V_{i_{k}}\textbf{Id}\big\|_{\text{Lip}(1)}\left\|\bold{X}_{[0,t]}\right\|_{p-var}^{k}.
\end{align*}
Unfortunately, this will not produce the desired estimate of $O(\epsilon)$ due to the Lipschitz norm $\|V_{i_{1}}\cdots V_{i_{k}}\textbf{Id}\big\|_{\text{Lip}(1)}$ being asymptotically larger than the uniform norm $\|V_{i_{1}}\cdots V_{i_{k}}\textbf{Id}\big\|_{\infty}$.
As a result, over the interval $[0, \epsilon^{2+\delta}]$, we can obtain an improved error estimate simply by showing that both $Z_{t}$ are $\mathcal{E}_{(V)}^{r}(Z_0, \bold{X}_{0,t})$ have size $O(\epsilon)$ and are therefore well behaved.
\end{remark}

\subsection{$L^{2}(\mathbb{P})$ error analysis}

We will now sate and prove an $L^{2}(\mathbb{P})$ version of the main result.
\begin{theorem}\label{theoremTaylorl2}
The one-step $r$-truncated Taylor approximation for the backward Loewner differential equation started from $Z_0 = \epsilon i$, with $\epsilon>0$ sufficiently small, admits an $O(\epsilon)$ error in an $L^{2}(\mathbb{P})$ sense over the time horizon $0 \leq t \leq \epsilon^{2}$. In other words, we have that
\begin{align*}
\left\|\,Z_{t} - \mathcal{E}^{r}_{(V)}(Z_0, \bold{X}_{0,t})\,\right\|_{L^{2}(\mathbb{P})}\leq C\epsilon,\hspace{2.5mm}\forall t\in\big[0,\epsilon^{2}\,\big],
\end{align*}
where $C$ is a finite constant that depends only on $r$. 
\end{theorem}
\begin{proof}
Using the same strategy of proof as for Theorem \ref{theoremTaylor}, it is enough to argue that:
\begin{align}
\big\|\bold{X}_{0,t}^{I}\big\|_{L^{2}(\mathbb{P})} & = \big\|\bold{X}_{0,1}^{I}\big\|_{L^{2}(\mathbb{P})}\,t^{\deg_{(1,2)}(I)},\label{integralscale}\\[3pt]
\big\|Z_{t}\|_{L^{2}(\mathbb{P})} & \leq C\epsilon,\label{solscale}
\end{align}
for small $0\leq t\leq\epsilon^{2}$. To show (\ref{integralscale}), we shall use the scaling property of Brownian motion:
\begin{align*}
\{B_{s}\}_{s\geq 0}\sim \left\{\frac{1}{\sqrt{c}}B_{cs}\right\}_{s\geq 0},
\end{align*}
where $c$ is constant and ``\,$\sim$\,'' means that both stochastic processes have the same law.
Therefore (\ref{integralscale}) follows by changing variables $\big(s_{i}\mapsto\tilde{s}_{i}:=\frac{1}{t}s_{i}\big)$ in the integral $\bold{X}_{0,t}^{I}$ so that
\begin{align*}
\left\|\,\int_{0<s_1<\cdots<s_k<t}dX_{s_1}^{i_1}\cdots dX_{s_k}^{i_k}\,\right\|_{L^{2}(\mathbb{P})} =\,\,\, \left\|\,t^{\deg_{(1,2)}(I)}\int_{0<\tilde{s}_1<\cdots<\tilde{s}_k<1}dX_{\tilde{s}_1}^{i_1}\cdots dX_{\tilde{s}_k}^{i_k}\,\right\|_{L^{2}(\mathbb{P})}.
\end{align*}
Note that $\big\|\bold{X}_{0,1}^{I}\big\|_{L^{2}(\mathbb{P})} < \infty$ follows by the It\^{o} isometry and triangle inequality for integrals. 
Similarly, we can estimate the $L^{2}(\mathbb{P})$ norm of the solution $Z_{t}$ by
\begin{align*}
\big\|Z_{t}\big\|_{L^{2}(\mathbb{P})} & = \left\|\,\int_{0}^{t}-\frac{2/\kappa}{Z_{s}}\,ds + B_{t}\,\right\|_{L^{2}(\mathbb{P})}\\[3pt]
& \leq \left\|\frac{2}{\kappa}\int_{0}^{t}\frac{1}{Z_{s}}\,ds\right\|_{L^{2}(\mathbb{P})} + \left\|B_{t}\right\|_{L^{2}(\mathbb{P})}\\[3pt]
& \leq \frac{2}{\kappa}\int_{0}^{t}\left\|\frac{1}{Z_{s}}\right\|_{L^{2}(\mathbb{P})}\,ds + \sqrt{t}\\[3pt]
& \leq \frac{2}{\kappa}\cdot\frac{1}{\epsilon}\cdot t + \sqrt{t}\\[3pt]
& \leq \left(\frac{2}{\kappa} + 1\right)\epsilon\,.
\end{align*}
\end{proof}

\section{SLE simulation using the Ninomiya-Victoir splitting}

In order to simulate an SLE trace, we must first discretize the backward Loewner equation,
\begin{align}\label{backwardloewner}
dZ_{t} & = -\frac{2}{Z_{t}}\,dt + \sqrt{\kappa}\,dB_{t}\,,\\
Z_{0} & = \epsilon i\,,\nonumber
\end{align}

Since the above SDE gives an explicit solution in the zero noise case (i.e. when $\kappa = 0$),
it is natural to apply a splitting method to approximate its solution. Moreover, as (\ref{backwardloewner})
can be viewed in Stratonovich form, such a method can be interpreted as the solution
of an ODE\,/\,RDE governed by the same vector fields but driven by a piecewise linear path.\\
Unfortunately, the convergence results of \cite{shekhar2019remarks} are not applicable if this has vertical pieces.
A well-known splitting method for (Stratonovich) SDEs is the Ninomiya-Victoir scheme, originally proposed in \cite{ninomiyavictoir2008origin}, which in our setting directly corresponds to the Strang splitting.

\begin{definition}[Ninomiya-Victoir scheme for SDEs driven by a single Brownian motion]
Consider an $n$-dimensional Stratonovich SDE on the interval $[0, T]$ with the following form
\begin{align}\label{generalsde}
dY_{t} & = V_{0}(Y_{t})\,dt  + V_{1}(Y_{t})\circ dB_{t}\,,\\
Y_{0} & = \xi\,,\nonumber
\end{align}
where $\xi\in\mathbb{R}^{n}$ and the vector fields $V_{i} : \mathbb{R}^{n} \rightarrow \mathbb{R}^{n}$ are assumed to be Lipschitz continuous.
For $t\in\mathbb{R}$ and $x\in\mathbb{R}^{n}$, let  $\,\exp(tV_{i})\hspace{0.5mm}x$ denote the unique solution at time $u=1$ of the ODE
\begin{align*}
\frac{dy}{du} & = tV_{i}(y)\,,\\
y(0) & = x\,.
\end{align*}
For a fixed number of steps $N$ we can construct a numerical solution $\big\{\widetilde{Y}_{t_{k}}\big\}_{0\leq k \leq N}$ of (\ref{generalsde}) by setting $\widetilde{Y}_{0} := \xi$ and for each $k \in [0 \mathrel{{.}\,{.}}\nobreak N-1]$, defining $\widetilde{Y}_{t_{k+1}}$ using a sequence of ODEs:
\begin{align}\label{nvstep}
\widetilde{Y}_{t_{k+1}} := \exp\left(\frac{1}{2}hV_{0}\right)\exp\Big(B_{t_{k}, t_{k+1}}V_{1}\Big)\exp\left(\frac{1}{2}hV_{0}\right)\widetilde{Y}_{t_{k}}\,,
\end{align}
where $h := \frac{T}{N}$ and $t_{k} := kh$.
\end{definition}

It was shown by Bally and Rey in \cite{ballyrey2016weakconv} that if the SDE (\ref{generalsde}) has smooth bounded vector
fields satisfying an ellipticity condition, then the Ninomiya-Victoir scheme converges in
total variation distance with order 2. That is, for $t\in(0,T]$ there exists $C_{t} < \infty$ such that
\begin{align*}
\forall N\hspace{-0.25mm} \geq 1,\, \forall f:\mathbb{R}^{n}\rightarrow\mathbb{R} \,\,\text{measurable and bounded},\, \sup_{t_{k}\geq t}\left|\mathbb{E}\Big[f\big(\widetilde{Y}_{t_{k}}\big)\Big]\hspace{-0.5mm} - \mathbb{E}\Big[f\big(Y_{t_{k}}\big)\Big]\right| \leq \frac{C_{t}\|f\|_{\infty}}{N^{2}}\hspace{0.25mm}.
\end{align*}

Furthermore, the strong convergence properties of this scheme were surveyed in \cite{gjc2016strongconv}.
Since the SDE (\ref{generalsde}) satisfies a commutativity condition, it was shown under fairly weak assumptions that the 
Ninomiya-Victoir scheme converges in an $L^{p}(\mathbb{P})$ sense with order 1:
\begin{align*}
\text{For}\,\, p\geq 2,\, \text{there exists}\,\, C > 0\,\,\,\text{such that for all}\,\, N \geq 1,\, \mathbb{E}\left[\,\sup_{t\in[0,T]}\big\|Y_{t} - \widetilde{Y}_{t}\big\|^{p}\right] \leq \frac{C}{N^{p}}\,,
\end{align*}
where the approximation $\widetilde{Y}_{t}$ is obtained by interpolating between the discretization points,
\begin{align*}
\widetilde{Y}_{t} := \xi + \frac{1}{2}\int_{0}^{t}V_{0}\big(\widetilde{Y}_{s}^{(0)}\big)\,ds + \int_{0}^{t}V_{1}\big(\widetilde{Y}_{s}^{(1)}\big)\circ dB_{s} + \frac{1}{2}\int_{0}^{t}V_{0}\big(\widetilde{Y}_{s}^{(2)}\big)\,ds\,,
\end{align*}
with the three (piecewise) processes $\widetilde{Y}^{(i)}$ defined over each interval $[\hspace{0.25mm}t_{k}, t_{k+1}\hspace{0.25mm}]$ according to
\begin{align*}
\widetilde{Y}_{t}^{(0)} & := \exp\left(\frac{1}{2}(t-t_{k})V_{0}\right)\widetilde{Y}_{t_{k}}\,,\\[3pt]
\widetilde{Y}_{t}^{(1)} & := \exp\Big(B_{t, t_{k}}V_{1}\Big)\,\widetilde{Y}_{t_{k+1}}^{(0)}\,,\\[3pt]
\widetilde{Y}_{t}^{(2)} & := \exp\left(\frac{1}{2}(t-t_{k})V_{0}\right)\widetilde{Y}_{t_{k+1}}^{(1)}\,.
\end{align*}

Turning our attention back to the Loewner differential equation (\ref{backwardloewner}), we see that the
imaginary part of the solution $\text{Im}(Z_{t})$ is increasing for all $t\geq 0$. So provided  $\text{Im}(Z_{0}) > 0$,
the $dt$ vector field becomes smooth and bounded on the domain $\left\{z\in\mathbb{C} : \text{Im}(z) \geq \text{Im}(\epsilon)\right\}$.
Moreover, this argument also shows that the derivatives of the $dt$ vector field are bounded.
Since the $dB_{t}$ vector field is constant, it will satisfy the various regularity assumptions
(including the ellipticity condition in \cite{ballyrey2016weakconv}). Hence, when applied to the backward Loewner equation,
the Ninomiya-Victoir scheme converges with the above strong and weak rates.
In particular, this implies the numerical method achieves a high order of weak convergence.
Another key feature is that the ODEs required to compute (\ref{nvstep}) can be resolved explicitly.
\begin{theorem}
When $V_{0}(z) = -\frac{2}{z}$ and $V_{1}(z) = \sqrt{\kappa}$ for $z\in\mathbb{C}$, we can explicitly show that
\begin{align*}
\exp(tV_{0})\hspace{0.5mm}z & = \sqrt{z^{2} - 4t}\,,\\
\exp(tV_{1})\hspace{0.5mm}z & = z + \sqrt{\kappa}\, t\,.
\end{align*}
\end{theorem}
Therefore, the proposed high order numerical method for discretizing (\ref{backwardloewner}) is given by

\begin{definition}[Ninomiya-Victoir splitting of the backward Loewner equation]\label{nvscheme}
For a fixed number of steps $N$, we construct a numerical solution $\big\{\widetilde{Z}_{t_{k}}\big\}_{0\leq k \leq N}$ of (\ref{backwardloewner}) on $[0,T]$ by setting $\widetilde{Z}_{0} := \epsilon i$ and for each $k \in [0 \mathrel{{.}\,{.}}\nobreak N-1]$, defining $\widetilde{Z}_{t_{k+1}}$ using the below formula, 
\begin{align}\label{nvloewnerstep}
\widetilde{Z}_{t_{k+1}} :=  \sqrt{\left(\sqrt{\widetilde{Z}_{t_{k}}^{2} - 2h_k} + \sqrt{\kappa}B_{t_{k},t_{k+1}}\right)^{2} - 2h_k}\,,
\end{align}
where $\{0 = t_0 < t_1 < \cdots < t_N = T\}$ is a partition of $[0,T]$ and $h_k := t_{k+1}-t_k$.
\end{definition}

A surprising property is that the scheme preserves the second and fourth moments. 
\begin{theorem}
Let $\widetilde{Z}$ denote the numerical approximation of $Z$ given by Definition \ref{nvscheme},
where $Z$ is the true solution of (\ref{backwardloewner}) with the initial condition $Z_0 = \widetilde{Z}_{0}$. Then for $k\geq 0$,
\begin{align*}
\mathbb{E}\big[\widetilde{Z}_{t_k}^2\big] & = \mathbb{E}\big[Z_{t_k}^2\big],\\
\mathbb{E}\big[\widetilde{Z}_{t_k}^4\big] & = \mathbb{E}\big[Z_{t_k}^4\big].
\end{align*} 
\end{theorem}
\begin{proof}
As Brownian motion has independent increments, it follows directly from (\ref{nvloewnerstep}) that
\begin{align*}
\mathbb{E}\big[\widetilde{Z}_{t_{k+1}}^2\big] = \mathbb{E}\big[\widetilde{Z}_{t_k}^2\big] + (\kappa - 4)h_k,
\end{align*}
and so $\mathbb{E}\big[\widetilde{Z}_{t_k}^2\big] = \mathbb{E}\big[\widetilde{Z}_{0}^2\big] + (\kappa - 4)t_k$ for $k \in [0 \mathrel{{.}\,{.}}\nobreak N]$. On the other hand, It\^{o}'s lemma gives
\begin{align*}
d\big(Z_{t}^{2}\,\big) = (\kappa - 4)\,dt + 2Z_{t}\sqrt{\kappa}\,dB_t.
\end{align*}
Therefore
\begin{align*}
\mathbb{E}\big[Z_{t_{k}}^2\big] = \mathbb{E}\big[Z_{0}^2\big] +(\kappa - 4)t_{k} +  2\sqrt{\kappa}\,\mathbb{E}\bigg[\int_0^{t_k}Z_u\,dB_u\,\bigg].
\end{align*}
The first result follows as the above It\^{o} integral of $Z$ against $B$ will have zero expectation.
To see that the fourth moments of $Z$ and $\widetilde{Z}$ are identical, we note that by It\^{o}'s isometry
\begin{align*}
&\mathbb{E}\big[Z_{t_{k+1}}^4\big]\\
& = \mathbb{E}\bigg[\bigg(Z_{t_k}^2 + (\kappa - 4)h_k + 2\sqrt{\kappa}\int_{t_k}^{t_{k+1}}Z_u\,dB_u\bigg)^2\,\bigg]\\
& = \mathbb{E}\Big[\big(Z_{t_k}^2 + (\kappa - 4)h_k\big)^2\Big] + 4\sqrt{\kappa}\,\mathbb{E}\bigg[\big(Z_{t_k}^2 + (\kappa - 4)h_k\big)\int_{t_k}^{t_{k+1}}\hspace{-0.5mm}Z_u\,dB_u\bigg] + 4\kappa\hspace{-0.5mm}\int_{t_k}^{t_{k+1}}\hspace{-0.5mm}\E\big[Z_u^2\big]\,du\\
& = \mathbb{E}\big[Z_{t_k}^4\big] + 2(\kappa - 4)h_k\mathbb{E}\big[Z_{t_k}^2\big] + (\kappa - 4)^2 h_k^2 + 4\kappa\int_{t_k}^{t_{k+1}}\big(\E\big[Z_{t_k}^2\big] + (\kappa - 4)(u-t_k)\big)\,du\\
& = \mathbb{E}\big[Z_{t_k}^4\big] + (6\kappa - 8)h_k\mathbb{E}\big[Z_{t_k}^2\big] + (\kappa - 4)^2 h_k^2 + 2\kappa(\kappa - 4)h_k^2\,.
\end{align*}
On the other hand, using (\ref{nvloewnerstep}), it is straightforward to compute the fourth moment of $\widetilde{Z}$.
\begin{align*}
\mathbb{E}\big[\widetilde{Z}_{t_{k+1}}^4\big] & = \mathbb{E}\bigg[\bigg(\widetilde{Z}_{t_{k}}^{2}  - 4 h_k - 2\sqrt{\kappa}B_{t_{k},t_{k+1}}\sqrt{\widetilde{Z}_{t_{k}}^{2} - 2h_k} + \kappa B_{t_{k},t_{k+1}}^2\,\bigg)^{\hspace{-0.5mm} 2}\,\bigg]\\
& = \mathbb{E}\big[\widetilde{Z}_{t_{k}}^4\big] - 8h_k\mathbb{E}\big[\widetilde{Z}_{t_{k}}^2\big] +  16 h_k^2 + 4\kappa h_k\big(\mathbb{E}\big[\widetilde{Z}_{t_{k}}^2\big] - 2h_k\big)\\
&\hspace{16.5mm} + 2\kappa h_k\big(\mathbb{E}\big[\widetilde{Z}_{t_{k}}^2\big] - 4 h_k\big) + \kappa^2\hspace{0.25mm}\mathbb{E}\big[B_{t_k, t_{k+1}}^4\big]\\
& = \mathbb{E}\big[\widetilde{Z}_{t_{k}}^4\big]  + (6\kappa - 8)h_k\mathbb{E}\big[\widetilde{Z}_{t_k}^2\big] + 16 h_k^2 - 16\kappa h_k^2 + 3\kappa^2 h_k^2\,.
\end{align*}
The result now follows as $Z$ and $\widetilde{Z}$ have the same initial value and second moments. 
\end{proof}
\begin{remark}
These properties are especially appealing as they hold on any time horizon. 
\end{remark}

To simulate SLE traces we shall incorporate the above numerical scheme into the
adaptive step size methodology proposed in \cite{kennedy2009slesim}. That is, instead of ``tilted'' or ``vertical''
slits, we use the Ninomiya-Victoir scheme described above to approximate the SLE trace.\medbreak

$SLE_{\kappa}$ traces can be built from conformal maps $g_{t}$ given by forward Loewner's equation,
\begin{align}\label{actualloewner}
\frac{dg_{t}(z)}{dt} & = \frac{2}{g_{t}(z) - \sqrt{\kappa}\, B_{t}}\,,\\
g_{0}(z) & = z\,.\nonumber
\end{align}

The $SLE$ curve $\gamma(t)$ is then defined to have the property that $g_{t}(\gamma(t)) = \sqrt{\kappa}\, B_{t}$ for $t\geq 0$.
Therefore, after applying the change of variables $h_{t} = g_{t} - \sqrt{\kappa}\,B_{t}$, we see that $\gamma(t) = \lim_{y \to 0}h_{t}^{-1}(iy)$
where
\begin{align}\label{shiftedloewner}
dh_{t}(z) & = \frac{2}{h(z)}dt - \sqrt{\kappa}\,dB_{t},\\
h_{0}(z) & = z\,.\nonumber
\end{align}

The backward Loewner equation (\ref{backwardloewner}) on $[0,T]$ generates a curve $\tilde{\gamma}(t)$ that modulo a shift with $\sqrt{\kappa}B_T$ has the same law with the SLE trace $\gamma(t)$ (see \cite{rohde2016backward}). In addition, we can use the Ninomiya-Victoir scheme (\ref{nvloewnerstep}) to approximate the backward Loewner diffusion.
For simulations, the challenge is that the driving Brownian motion must be run backwards.
More concretely, if we fix a partition $0 = t_{0} < t_{1} < \cdots < t_{N} = T$,
then we can construct a numerical SLE trace $\big\{z_{t_{k}}\big\}_{0\leq k \leq N}$ by setting $z_{0} := 0$ and for $k \in [1 \mathrel{{.}\,{.}}\nobreak N]$ defining $\tilde{z}_{t_{k}}$ by
\begin{align}\label{numericalloewner}
z_{t_{k}} := f_{0}\circ f_{1} \circ \cdots \circ f_{k-1} (0)\,,
\end{align} 
where
\begin{align*}
f_{i}(z) :=  \sqrt{\left(\sqrt{z^{2} - 2(t_{i+1}-t_{i})} + \sqrt{\kappa}B_{t_{i+1}, t_{i}}\right)^{2} - 2(t_{i+1}-t_{i})}\,.
\end{align*}\smallbreak

As discussed in \cite{kennedy2009slesim}, due to the singularity at 0 inherent in the conformal maps $g_{t}$,
simulating SLE traces using a fixed uniform partition can lead to huge numerical errors.
Instead, an adaptive step size methodology was recommended, especially when $\kappa$ is large.
The idea is to ensure that $\big|z_{t_{k+1}} - z_{t_{k}}\big| < C$ for each $k$, where $C$ is a user-specified tolerance.
To achieve this, we follow precisely the same adaptive step size strategy as proposed in \cite{kennedy2009slesim}. That is, we start by computing $\{z_{t_k}\}_{k\hspace{0.25mm}\geq\hspace{0.25mm} 0}$ along a uniform partition until $\big|z_{t_{k+1}} - z_{t_{k}}\big| \geq C$.
If this occurs, it indicates that we should reduce the step size for the SLE discretization. Therefore, we shall sample the Brownian path at the midpoint of the interval $[t_{k}, t_{k+1}]$.
(This can be done using a Brownian bridge conditioned on the values of $B$ at $t_{k}$ and $t_{k+1}$)
We now proceed as before, except we have added the midpoint of $[t_{k}, t_{k+1}]$ to the partition.
This process continues (recursively) until each value of $\big|z_{t_{k+1}} - z_{t_{k}}\big|$ is strictly less than $C$.
\begin{figure}[h]
\begin{center}
\includegraphics[width=0.96\textwidth]{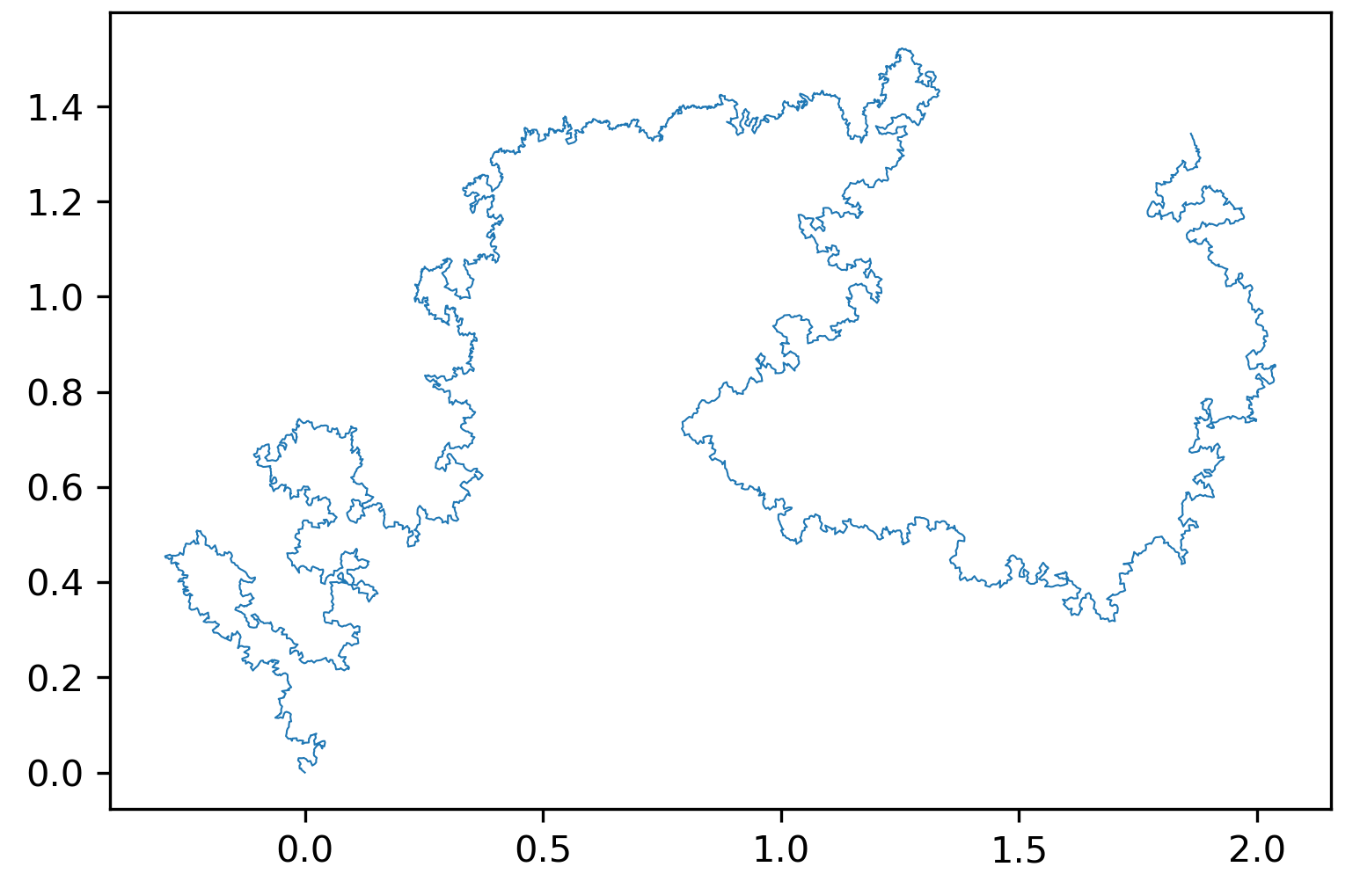}
\caption{A sample of $SLE_{\kappa}$ trace for $\kappa=8/3$ with 5103 points.}
\end{center}\vspace{-8.5mm}
\end{figure}

Figures 6.1 and 6.2 demonstrate that the proposed numerical method can generate realistic simulations of the $SLE_\kappa$ trace, even for larger values of $\kappa$. In Section 4 of \cite{tran2015convergence}, the author claims that the vertical slit method proposed by Kennedy \cite{kennedy2009slesim} converges to the $SLE_\kappa$ trace when fixed step sizes are used. Since each step of the Ninomiya-Victoir scheme is a composition of two vertical slits, we also expect such a convergence result.
However in practice, one observes significantly improved performance when adaptive step sizes are used \cite{kennedy2009slesim}. Therefore in our final section, we shall be considering the latter setting.\vspace{-5mm}

\begin{figure}[H]
\begin{center}
\includegraphics[width=0.975\textwidth]{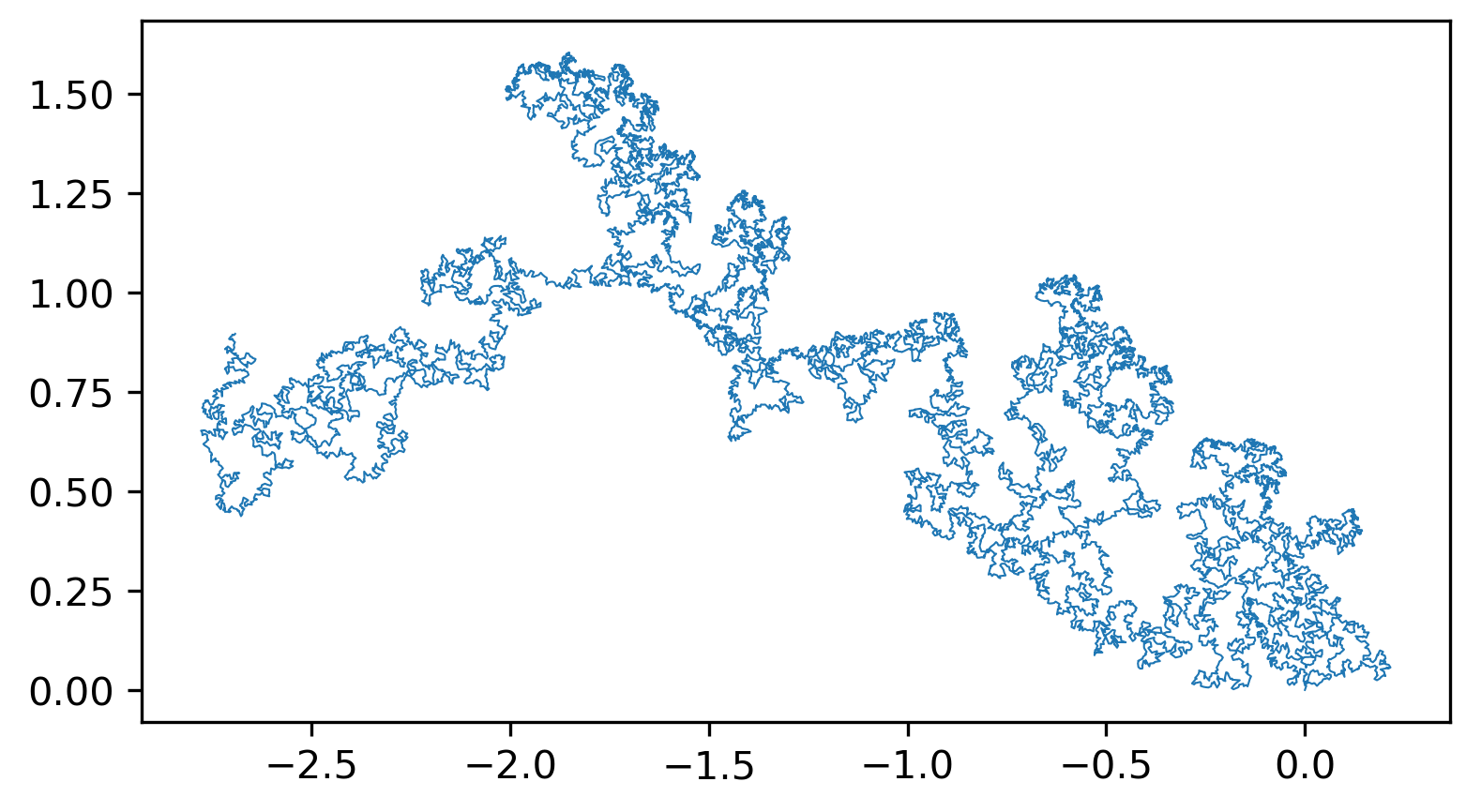}
\caption{A sample of $SLE_{\kappa}$ trace for $\kappa=6$ with 17884 points.}
\end{center}
\end{figure}\vspace{-14mm}
This approach naturally leads to the open problem of whether alternative high order 
``ODE-based'' methods can be applied to SLE simulation (such as those presented in \cite{foster2020odemethod}).

\subsection{Error analysis of the scheme}

In this subsection, we will perform an error analysis of the Ninomiya-Victoir scheme to the SLE trace when $\kappa\neq 8$.  To this end, we define the constants $a=2/\kappa$, $r=a+(1/4) < 2a+1$ and 
$b=\frac{(1+2a)r-r^2}{a}=a+1+\frac{3}{16a}$. 
\begin{proposition}[Proposition $7.3$ in \cite{lawler2008conformally}]\label{coro}
For every $0 \leq r \leq 2a+1\,,$ there is a finite $c=c(a,r)$ such that for all $0 \leq t \leq 1$, $0 \leq y_0 \leq 1\,,$ $e \leq \lambda \leq y_0^{-1}\,,$ we have that
$$\mathbb{P}\big(|Z_{t}^\prime(z_0)| \geq \lambda\big) \leq \lambda^{-b}(|z_0|/y_0)^{2r}\delta(y_0, \lambda)\,, $$
where $b=[(2a+1)r-r^2]/a \geq 0$ and
\[
    \delta(y_0, \lambda)= 
\begin{cases}
    \lambda^{(r/a)-b},& \text{if }   r < ab\,,\\
    -\log (\lambda y_0), & \text{if} \hspace{1mm} r = ab\,,\\
    y_0^{b-(r/a)}, & \text{if }  r > ab\,.
\end{cases}
\]
\end{proposition}
So in the regime $r <ab\,,$ by Proposition \ref{coro}, there exists $\epsilon > 0 $ such that
\begin{align*}
\mathbb{P} \big(|Z'_t(i2^{-j})| \geq 2^{j-\epsilon}\big) \leq c2^{-j(2b-(r/a))(1-\epsilon)}\,,
\end{align*}
for $j\geq 1$. Thus, we can apply Borel-Cantelli argument provided that $a \neq 1/4$, i.e. $\kappa \neq 8$, to obtain that for all values of $t \in [0,T]$ (by scaling is enough to look only for $t \in [0,1]$) 
we have that almost surely:
\begin{align*}
\big|Z^\prime_{k2^{-2j}}(i2^{-j})\big| \leq c2^{j-\epsilon}\,,
\end{align*}
for $j\in\{1,2,\ldots\}$ and  $k\in\{0,1, \ldots, 2^{2j}\}$.
%By choosing for example $k=2^{2j}$, we obtain an estimate on the derivative of the map at time $t=1$.

The result gives the estimate on the derivative of the conformal map $h_t(z)$ at the dyadic times in the interval $[0,1]$. Using the Distortion Theorem (see Lemma $2.2$ in \cite{viklund2014continuity}) we extend the result for all the points. Moreover, using the estimate on the derivative of the map, we obtain that for almost every Brownian path, $|\tilde{\gamma}(t)-Z_t(iy)|\leq \int_0^y |Z^\prime_t(ir)|dr$ for $t \in [0, 1]$, where $\tilde{\gamma}(t)=\lim_{y \to 0}Z_t(iy)$. In particular, this leads to the following lemma.

\begin{lemma}\label{approx_for_trace_thm} There exists $\epsilon > 0$ such that
\begin{align}\label{approx_for_trace}
\sup\limits_{t\in[0,1]}\big|\tilde{\gamma}(t)-Z_t(iy)\big|\leq C(\omega) y^{1-\epsilon}
\end{align}
where the constant depends on $\omega$ but is finite almost surely.
\end{lemma}

It is now worth nothing that for any $y>0$, the dynamics of $t\mapsto Z_t(iy)$ is given by an SDE whose vector fields are smooth, bounded and with bounded derivatives since $\text{Im}(Z_t(iy))$ is increasing in $t$. Most notably, this gives the simple estimate $\text{Im}(Z_t(iy)) \geq y$.
Any discretized process $\{\widetilde{Z}_{t_k}(iy)\}_{k\geq 0}$ obtained via the Ninomiya-Victoir scheme will also enjoy this lower bound on its imaginary part as
\begin{align*}
\text{Im}(z) & \leq \text{Im}\left(\sqrt{z^2 - c}\,\right),\\
\text{Im}(z) & = \text{Im}(z+c),
\end{align*}
for $z\in\mathbb{H}$ and $c \geq 0$. So when $y>0$, we can apply standard results from the SDE literature to $Z$ and $\widetilde{Z}$, which will establish convergence for the Ninomiya-Victoir scheme.
In particular, the results of \cite{lyonsgaines1997variable} give theoretical guarantees for choosing step sizes based on the Brownian path itself whereas previous results on the convergence of algorithms simulating Loewner curves only consider predetermined and uniform time stepping \cite{tran2015convergence}. Moreover, it was shown in \cite{kennedy2009slesim} that adaptive steps can significantly improve performance for SLE simulation. Whilst our analysis concerns an adaptive Ninomiya-Victoir scheme, the following two theorems may also extend to the adaptive power series approximation proposed by Tom Kennedy in \cite{kennedy2009slesim}, where convergence results were not formally established.

\begin{theorem}\label{variable_step_thm} For almost all $\omega\in\Omega$ and any $\epsilon, y > 0$, there exists $\delta > 0$ such that for every partition $\mathcal{D} = \{0 = t_0 < t_1 < \cdots < t_N = 1\}$ of $[0,1]$ with $\,\mesh\mathcal{D} \leq \delta$, we have
\begin{align*}
\sup_{t_k\in\mathcal{D}}\big|Z_{t_k}(iy) - \widetilde{Z}_{t_k}(iy)\big| \leq \epsilon,
\end{align*}
where $\{\widetilde{Z}_{t_k}(iy)\}_{k\geq 0}$ denotes the Ninomiya-Victoir discretization of $Z$ started at $\widetilde{Z}_0 = iy$ and computed using the points in $\mathcal{D}$.
\end{theorem}
\begin{proof}
As noted previously, the backward Loewner equation (\ref{backwardloewner}) has smooth and bounded vector fields whenever $Z_0 = iy$ with $y>0$. Moreover, both the diffusion $Z$ and its discretization $\widetilde{Z}$ lie in the domain $\{z\in\mathbb{C} : \text{Im}(z) \geq y\}$. Finally, we note that over an interval $[s,t]$ of size $h = t-s$, the Ninomiya-Victoir scheme admits the Taylor expansion:
\begin{align}\label{nvtaylor}
\widetilde{Z}_t = \widetilde{Z}_s - \frac{2}{\widetilde{Z}_{s}}\,h + \sqrt{\kappa}B_{s,t} + o(h),
\end{align}
provided $s$ and $t$ are sufficiently close together. For general SDEs driven by a single Brownian motion, this expansion would have a $B_{s,t}^2$ term preceding the $o(h)$ remainder \cite{ninomiyavictoir2008origin}.
The result now directly follows from (\ref{nvtaylor}) using Theorem 4.3 and Corollary 4.4 in \cite{lyonsgaines1997variable}.
\end{proof}
\begin{remark}
When $|\widetilde{Z}_s|$ is small, we expect 
a step size of $O(|\widetilde{Z}_s|^{2+})$ to give an accurate approximation of the true diffusion started from $\widetilde{Z}_s$ (by our main result, Theorem \ref{theoremTaylor}). 
\end{remark}
We can now establish a global error estimate for the Ninomiya-Victoir scheme (\ref{nvloewnerstep}).
\begin{theorem}
For almost all $\omega\in\Omega$ and any $\epsilon > 0$, there exists $y > 0$ and $\delta > 0$ such that for every partition $\mathcal{D} = \{t_0 < t_1 < \cdots < t_N\}$ of $[0,1]$ with $\,\mesh\mathcal{D} \leq \delta$, we have
\begin{align*}
\sup_{t_k\in\mathcal{D}}\big|\tilde{\gamma}(t) - \widetilde{Z}_{t_k}(iy)\big| \leq \epsilon.
\end{align*}
\end{theorem}
\begin{proof}
The result is an immediate consequence of Lemma \ref{approx_for_trace_thm} and Theorem \ref{variable_step_thm}.
\end{proof}
\begin{remark}
We expect the discretized process $\widetilde{Z}$ to satisfy an estimate similar to (\ref{approx_for_trace}). Such a result would then establish convergence for the process $\widetilde{Z}(0)$. This conjecture is supported by numerical evidence where traces were computed using an initial value of $0$.
\end{remark}

\bibliographystyle{plain}
\bibliography{references}

\end{document}